\documentclass[10pt]{amsart}
\usepackage{mathrsfs}
\usepackage{epsfig}
\usepackage[dvipsnames]{xcolor}
\usepackage{graphicx}
\usepackage{amsmath}
\usepackage{amsfonts}
\usepackage{amssymb}
\usepackage{bbm}
\usepackage{bm}
\usepackage{amsthm}
\usepackage{amscd}
\usepackage{verbatim}
\usepackage{tikz}
\usepackage[sc]{mathpazo}
\usepackage[T1]{fontenc} 
\usepackage[autostyle]{csquotes}
\usepackage{ytableau}
\usepackage{enumitem}
\usetikzlibrary{arrows.meta}

\usepackage[hidelinks]{hyperref}
\usepackage[nameinlink,noabbrev]{cleveref}

\tikzset{
vertex/.style={circle,draw,minimum size=1.5em},
edge/.style={->,> = latex'}
}

\usepackage{mathtools}

\theoremstyle{plain}
\newtheorem{theorem}{Theorem}[section]

\newtheorem{lemma}[theorem]{Lemma}

\newtheorem{conjecture}[theorem]{Conjecture}
\theoremstyle{definition}
\newtheorem{definition}[theorem]{Definition}

\newtheorem{remark}[theorem]{Remark}
\newtheorem{example}[theorem]{Example}

\numberwithin{equation}{section}

\DeclareMathAlphabet{\mathpzc}{OT1}{pzc}{m}{it}

\definecolor{MyBlue1}{RGB}{229, 119, 167}
\definecolor{MyBlue2}{RGB}{61, 101, 165}
\definecolor{MyBlue3}{RGB}{124, 161, 204}

\newcommand{\pr}{\mathbb{P}}

\newcommand{\norm}[1]{\| #1 \|}

\usepackage{marginnote}

\let\oldenumerate=\enumerate
\def\enumerate{
\oldenumerate
\setlength{\itemsep}{5pt}
}
\let\olditemize=\itemize
\def\itemize{
\olditemize
\setlength{\itemsep}{5pt}
}

\setlist[enumerate]{leftmargin=2.5em}

\begin{document}

\title[Stability of Khintchine-type inequalities via log-monotonicity]{Stability of Khintchine-type inequalities via log-monotonicity}
\author[Ch\'avez]{\'Angel Ch\'avez}
\author[Sheng]{Sam Sheng}

\address{Department of Mathematics and Computer Science, Davidson College, 209 Ridge Rd, Box 5000, Davidson, NC 28035} 
\email{anchavez@davidson.edu}

\address{Department of Economics, Columbia University, 1022 International Affairs Building, Mail Code 3308, 420 West 118th Street, New York, NY 10027} 
\email{zs2819@columbia.edu}

\begin{abstract}
We investigate Khintchine-type inequalities for the weighted sums $S=\sum_ka_kX_k$  of independent copies of a symmetric random variable $X$. We show how  log-monotonicity of the sequence $r_k(X)=k! \mathbb{E}[X^{2k}]/(2k)!$ implies sharp comparisons between the $L_p$ and $L_2$ norms of $S$ for every even integer $p\geq 2$, extending classic Khintchine-type inequalities and yielding new results in the log-convex setting. We also investigate the stability of our inequalities. Our first stability inequality sharpens the classic inequality by a deviation of the coefficient vector from the coordinate extremizers, while the second quantifies deviation from the Gaussian limit. Our results recover recent stability inequalities for random signs and apply to a broad class of distributions, including type-$\mathscr{L}$ random variables, ultra sub-Gaussian random variables and Gaussian mixtures. 
\end{abstract}

\maketitle

\section{Introduction}

\subsection{Foreword} The study of Khintchine inequalities was initiated by Aleksandr Khintchine \cite{Khintchine} and goes as follows. Suppose $\epsilon_1, \epsilon_2, \ldots, \epsilon_n$ are iid Rademacher random variables defined by $\pr(\epsilon_i=\pm 1)=\frac{1}{2}$. Consider the sum $S=\sum_ka_k\epsilon_k$, in which $a=(a_1,a_2, \ldots, a_n)\in \mathbb{R}^n$. Khintchine showed in \cite{Khintchine} that there exist constants $C_{p,q}>0$ such that for all $n\geq 1$ and coefficients $a\in \mathbb{R}^n$, we have 
\begin{align}
\norm{S}_p\leq C_{p,q}\norm{S}_q,\label{eq:Original}
\end{align} in which $\norm{\cdot}_p=(\mathbb{E}|\cdot|^p)^{\frac{1}{p}}$ denotes the $L_p$ norm. The problem of finding optimal constants $C_{p,q}$ has been studied for several decades and by several authors; see \cite{Figiel, Haagerup, Khintchine, Latala2, Mordhorst, Nayar1, Nazarov, Ole, Pinelis, Stechkin, Szarek,Whittle, Young}, for instance. In all known cases,
\begin{align*}
C_{p,q}=\min\big\{\textstyle\frac{\gamma_p}{\gamma_q}, \frac{\norm{Y}_p}{\norm{Y}_q}\big\},
\end{align*}where $\gamma_p$ denotes the $L_p$ norm of a standard Gaussian and $Y=\frac{1}{\sqrt{2}}(\epsilon_1+\epsilon_2)$.


In general, suppose $X_1, X_2, \ldots, X_n$ are iid copies of a random variable $X$. The study of Khintchine-type inequalities seeks to find a sharp comparison between the $L_p$ and $L_q$ norms of the sum $S=\sum_ka_kX_k$, where $a\in \mathbb{R}^n$. This problem has been studied for many distributions; see \cite{Eskenazis1, Eskenazis2, Havrilla2, Havrilla1, Latala1, Nayar1, Newman2, Newman1}, for instance. 

Sharp comparisons between the $L_p$ and $L_2$ norms of $S$ are recorded in Figure \ref{fig:Table1}, in which we denote $\eta_p=\norm{X}_p/\norm{X}_2$. The list is by no means exhaustive.

\begin{figure}[ht]
\begin{center}
\renewcommand{\arraystretch}{1.5}
\begin{tabular}{ |c|c|c|c| } 
\hline
$X$ & Optimal bounds & Regime & Source \\ 
\hline
$\mathcal{U}(-1,1)$& $\eta_p \norm{S}_2\leq\norm{S}_p\leq \gamma_p \norm{S}_2$& $p\geq 2$& \cite{Latala1} \\ 
Rademacher& $\eta_p \norm{S}_2\leq\norm{S}_p\leq \gamma_p \norm{S}_2$& $p\geq 2$& \cite{Haagerup} \\ 
Gaussian mixture& $\gamma_p\norm{S}_2\leq\norm{S}_p\leq \eta_p\norm{S}_2$ & $p\geq 2$ & \cite{Eskenazis1}\\ 
3-point symmetric& $\eta_p \norm{S}_2\leq\norm{S}_p\leq \gamma_p \norm{S}_2$ & $p\geq 3$& \cite{Havrilla1}\\ 
\hline
\end{tabular}
\end{center}
\caption{Sharp comparisons between the $L_p$ and $L_2$ norms of $S$.}
\label{fig:Table1}
\end{figure}

Our paper establishes new Khintchine-type inequalities of the above type. In particular, we obtain sharp comparisons between the $L_p$ and $L_2$ norms of the weighted sum $S=\sum_ka_kX_k$ for even $p\geq 2$. Furthermore, we investigate the stability of our inequalities. That is to say, we strengthen our inequalities by a defect which measures the distance of $a\in \mathbb{R}^n$ from the extremizer. We remark that the question about stability of Khintchine-type inequalities is relatively new but has already been addressed by several authors; see \cite{De, Eskenazis3, Eskenazis4, Jakimiuk, Jakimiuk2, Melbourne}.

\subsection{Szarek stability} 
For historical context, we recall the stability of the classic Khintchine inequality $\norm{S}_2\leq C_{2,1}\norm{S}_1$, where $S=\sum_ka_k\epsilon_k$. Recall the optimal constant appearing in \eqref{eq:Original} is $C_{2,1}=\frac{\norm{Y}_2}{\norm{Y}_1}$, where $Y=\frac{1}{\sqrt{2}}(\epsilon_1+\epsilon _2)$. This was originally established by Szarek in \cite{Szarek} and implies $\mathbb{E}|S|\geq \mathbb{E}|Y|$. The authors of \cite{De, Eskenazis3, Melbourne} sharpened Szarek's inequality by proving bounds of the form
\begin{align*}
\mathbb{E}|S|\geq \mathbb{E}|Y|+c\delta(a),
\end{align*}in which $\sum_ka_k^2=1$, $c>0$ is constant independent of $n$ and $\delta$ measures the distance of the vector $a\in \mathbb{R}^n$ from the two-point extremizers at $\frac{1}{\sqrt{2}}(\pm e_i\pm e_j)$.

\subsection{Coordinate stability} Our paper concerns symmetric distributions whose moment sequences, after normalization, are log-monotone. Log-monotonicity plays a central role in our paper. Recall a sequence of positive numbers $(r_k)$ is called \emph{log-concave} if $r_k^2\geq r_{k-1}r_{k+1}$ and \emph{log-convex} if $r_k^2\leq r_{k-1}r_{k+1}$. Log-monotone sequences arise often in algebra, combinatorics and geometry; see \cite{Branden, StanleyLog}.

\begin{definition}
Let $X$ be a symmetric random variable. Define the sequence
\begin{align*}
r_k(X)=\frac{ k!\mathbb{E}[X^{2k}]}{(2k)!}
\end{align*}for $k\geq 0$. For brevity, we denote $r_k=r_k(X)$ when there is no risk of confusion. 
\end{definition} 

We prove the following \emph{coordinate stability} theorem in Section \ref{sec:Robust1}. Observe equality is attained in Theorem \ref{thm:Robust1} by the one-point extremizers at $a=\pm e_i$.

\begin{theorem}\label{thm:Robust1}
Suppose $p=2q$ for some integer $q\geq 1$ and $(a_1, a_2, \ldots, a_n)\in \mathbb{R}^n$ satisfies $\sum_ka_k^2=1$. Furthermore, suppose $X$ is a symmetric random variable and $X_1, X_2, \ldots, X_n$ are iid copies of $X$. Let $r_k=r_k(X)$. Then $S=\sum_ka_kX_k$ satisfies the following bounds.
\begin{enumerate}
\item If $(r_k)$ is log-concave, then $r_1r_{q-1}-r_q\geq 0$ and
\begin{align*}
\norm{S}_p^p\geq \textstyle\Big(\frac{\norm{X}_p}{\norm{X}_2}\Big)^p\norm{S}_2^p+ \displaystyle 2^q\gamma_p^p(r_1r_{q-1}-r_{q})\Big(1- \sum_ka_k^p\Big);
\end{align*} 

\item If $(r_k)$ is log-convex, then $r_q-r_1r_{q-1}\geq 0$ and
\begin{align*}
\norm{S}_p^p\leq \textstyle\Big(\frac{\norm{X}_p}{\norm{X}_2}\Big)^p\norm{S}_2^p- \displaystyle 2^q\gamma_p^p(r_{q}-r_1r_{q-1})\Big(1- \sum_ka_k^p\Big).
\end{align*}  
\end{enumerate}Moreover, the constants are asymptotically sharp near the coordinate extremizers.
\end{theorem}

\subsection{Gaussian stability} We prove the following \emph{Gaussian stability} theorem in Section \ref{sec:Robust2}. We remark that equality is attained in Theorem \ref{thm:Robust2} asymptotically by $a=(\frac{1}{\sqrt{n}}, \frac{1}{\sqrt{n}}, \ldots,\frac{1}{\sqrt{n}})$. In this case, the central limit theorem ensures that
\begin{align*}
\lim_{n\to \infty} \norm{S}_p=\gamma_p\norm{X}_2=\gamma_p\norm{S}_2.
\end{align*} 

\begin{definition}
Let $q\geq 2$ be an integer. For $a\in\mathbb{R}^n$ satisfying $\sum_ka_k^2=1$, define
\begin{align*}
0\leq \Delta_q(a)=\begin{cases} 1& n<q,\\
1-q!\binom{n}{q}\Big( \frac{1-\sum_ka_k^4}{n(n-1)}\Big)^{\frac{q}{2}}& n\geq q.
\end{cases}
\end{align*}
\end{definition}

\begin{theorem}\label{thm:Robust2}
Suppose $p=2q$ for some integer $q\geq 2$ and $(a_1, a_2, \ldots, a_n)\in \mathbb{R}^n$ satisfies $\sum_ka_k^2=1$. Furthermore, suppose $X$ is a symmetric random variable and $X_1, X_2, \ldots, X_n$ are iid copies of $X$. Let $r_k=r_k(X)$. Then $S=\sum_ka_kX_k$ satisfies the following bounds.
\begin{enumerate}
\item If $(r_k)$ is log-concave, then $r_1^2-r_2\geq 0$ and
\begin{align*}
\norm{S}_p^p\leq \gamma_p^p\norm{S}_2^p-2^q\gamma_p^pr_1^{q-2}(r_1^2-r_2)\Delta_q(a);
\end{align*} 

\item If $(r_k)$ is log-convex, then $r_2-r_1^2\geq 0$ and
\begin{align*}
\norm{S}_p^p\geq \gamma_p^p\norm{S}_2^p+2^q\gamma_p^pr_1^{q-2}(r_2-r_1^2)\Delta_q(a).
\end{align*} 
\end{enumerate} Moreover, the constants are asymptotically sharp near the Gaussian extremizer.
\end{theorem}

We remark that $\Delta_q(a)\geq \sum_ka_k^4$ for all $n$. This bound is proved in Section \ref{sec:Robust2}. If $(r_k)$ is log-concave, then for all $n$ we must have
\begin{align*}
\norm{S}_p^p\leq \gamma_p^p\norm{S}_2^p-2^q\gamma_p^pr_1^{q-2}(r_1^2-r_2)\sum_ka_k^4.
\end{align*}Jakimiuk recently proved in \cite{Jakimiuk} that there exist constants $C_p>0$ for which 
\begin{align}
\norm{\sum_ka_k\epsilon_k}_p^p\leq \gamma_p^p-C_p\sum_ka_k^4\label{eq:Jakimiuk}
\end{align} for any $p\geq 4$ and vector $a\in \mathbb{R}^n$ satisfying $\sum_ka_k^2=1$. If $\epsilon$ is a Rademacher random variable, then $\mathbb{E}[\epsilon^{2k}]=1$ for all $k\geq 0$, which implies $r_k(\epsilon)=\frac{k!}{(2k)!}$ is log-concave. Therefore, Theorem \ref{thm:Robust2} implies \eqref{eq:Jakimiuk} for even integers $p$ with
\begin{align*}
C_p&=2^{\frac{p}{2}}\gamma_p^pr_1^{\frac{p}{2}-2}(r_1^2-r_2)\\
&=\textstyle2^{\frac{p}{2}}\gamma_p^p(\frac{1}{2})^{\frac{p}{2}-2}(\frac{1}{4}-\frac{1}{12})\\
&=\textstyle \frac{2}{3}\gamma_p^p.
\end{align*} 

Finally, we remark that the investigation of stability of Khintchine-type inequalities has been largely restricted to the Rademacher case. In fact, \cite{Jakimiuk2} appears to be among the first papers to consider other distributions. In particular, the authors of \cite{Jakimiuk2} sharpen moment comparison inequalities for sums of random vectors uniformly distributed on the $d$-dimensional Euclidean sphere $S^d\subset \mathbb{R}^{d+1}$. If $d=0$, then this reduces to the Rademacher case. To the best of our knowledge, our paper is also among the first investigations of stability phenomena for Khintchine-type inequalities beyond the Rademacher setting.

\subsection{Classic inequalities} We now state classic comparisons between the $L_p$ and $L_2$ norms of $S=\sum_ka_kX_k$. The following theorem, which nicely reflects the behavior seen in Figure \ref{fig:Table1}, follows immediately from Theorems \ref{thm:Robust1} and \ref{thm:Robust2}. We remark that an alternative proof of the following theorem is given in Section \ref{sec:Khintchine}.

\begin{theorem}\label{thm:Main}
Let $X$ be a symmetric random variable, $X_1, X_2, \ldots, X_n$ iid copies of $X$ and $r_k=r_k(X)$. Let $(a_1, a_2, \ldots, a_n)\in\mathbb{R}^n$ and $p\geq 2$ be an even integer. Then the following bounds hold for the sum $S=\sum_ka_kX_k$. Moreover, the constants are optimal.
\begin{enumerate}
\item If $(r_k)$ is log-concave, then $\big( \frac{\norm{X}_p}{\norm{X}_2}\big) \norm{S}_2\leq \norm{S}_p\leq \gamma_p\norm{S}_2$;\\

\item If $(r_k)$ is log-convex, then $\gamma_p\norm{S}_2\leq \norm{S}_p\leq \big( \frac{\norm{X}_p}{\norm{X}_2}\big) \norm{S}_2$.
\end{enumerate}
\end{theorem}

Log-monotonicity in the study of Khintchine inequalities is not new. Indeed, a symmetric random variable $X$ is \emph{ultra sub-Gaussian} if the sequence $(2^kr_k)$ is log-concave. These random variables were introduced by Nayar and Oleszkiewicz in \cite{Nayar1}. Observe $(r_k)$ is log-concave if $(2^kr_k)$ is log-concave. Therefore, ultra sub-Gaussian distributions satisfy the bounds in Theorems \ref{thm:Robust1}, \ref{thm:Robust2} and \ref{thm:Main}.

Following \cite{Newman1}, a symmetric random variable $X$ is of \emph{type-$\mathscr{L}$} if there exist constants $c,c'>0$ for which $|\mathbb{E}[e^{zX}]|\leq c\exp(c'|z|^2)$ for all $z\in \mathbb{C}$ and $\mathbb{E}[e^{zX}]$ is even with only purely imaginary zeros. Type-$\mathscr{L}$ random variables are shown to be ultra sub-Gaussian in \cite{Havrilla2}. Consequently, type-$\mathscr{L}$ random variables also satisfy the bounds in Theorems \ref{thm:Robust1}, \ref{thm:Robust2} and \ref{thm:Main}. We remark that the upper bound in part (a) of Theorem \ref{thm:Main} was established for type-$\mathscr{L}$ random variables by Newman \cite{Newman2,Newman1}. Additionally, Havrilla, Nayar and Tkocz recently proved in \cite{Havrilla2} that $\norm{S}_p\leq \big(\frac{\gamma_p}{\gamma_q}\big)\norm{S}_q$ for all even integers $2\leq q\leq p$. Their result holds for type-$\mathscr{L}$ random variables where the symmetry condition on $\mathbb{E}[e^{zX}]$ is omitted.

The authors believe the lower bound in (a) of Theorem \ref{thm:Main} and the entirety of part (b) of Theorem \ref{thm:Main} to be new results. We remark that the machinery used in \cite{Havrilla2, Nayar1, Newman2, Newman1} does not appear to apply in the log-convex case. As such, these references do not contain any examples of distributions $X$ for which $r_k(X)$ is log-convex. We give some examples of such distributions in Section \ref{sec:Examples}.

\subsection{Outline of the paper} Section \ref{sec:Monotone} introduces the Schur-monotone property of distributions. We show this property is enjoyed by distributions whose sequences $(r_k)$ are log-monotone. Historically, the Schur-monotone property for distributions has been used as a mechanism for deriving Khintchine-type inequalities. Section \ref{sec:Khintchine} describes this mechanism, which provides an alternate proof of Theorem \ref{thm:Main}. In Section \ref{sec:Examples}, we give examples of distributions for which $(r_k)$ is log-convex. We prove Theorem \ref{thm:Robust1} in Section \ref{sec:Robust1} and Theorem \ref{thm:Robust2} in Section \ref{sec:Robust2}. Finally, we pose a diagonal stability conjecture in Section \ref{sec:Diagonal}. 

Figure \ref{fig:Diagram1}, which is shown below, details the structure of our paper and shows how each concept is connected. Log-monotonocity serves as our cornerstone.

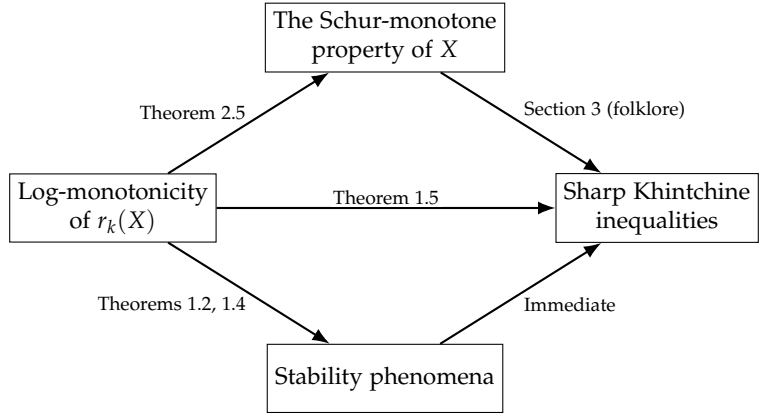
\begin{figure}[ht]
\begin{center}
\begin{tikzpicture}[
	scale=0.9,
      transform shape,
    >={Latex},
    box/.style={
        draw,
        rectangle,
        minimum width=2.8cm,
        minimum height=1cm,
        align=center
    },
    arrowlabel/.style={
        font=\footnotesize,
        fill=white,
        inner sep=1pt
    }
]

\node[box] (top)    at (0,2.5)  {The Schur-monotone\\ property of $X$};
\node[box] (right)  at (4,0)    {Sharp Khintchine\\ inequalities};
\node[box] (bottom) at (0,-2.5) {Stability phenomena};
\node[box] (left)   at (-4,0)   {Log-monotonicity\\ of $r_k(X)$};

\draw[->, thick]
(left) -- node[midway, below left, arrowlabel]
{Theorems \ref{thm:Robust1}, \ref{thm:Robust2}} (bottom);

\draw[->, thick]
(left) -- node[midway, above, arrowlabel]
{Theorem \ref{thm:Main}} (right);

\draw[->, thick]
(top) -- node[midway, above right, arrowlabel]
{Section \ref{sec:Khintchine} (folklore)} (right);

\draw[->, thick]
(bottom) -- node[midway, below right, arrowlabel]
{Immediate} (right);

\draw[->, thick]
(left) -- node[midway, above left, arrowlabel]
{Theorem \ref{thm:Main2}} (top);
\end{tikzpicture}
\end{center}
\caption{The structural diagram of our underlying concepts.}
\label{fig:Diagram1}
\end{figure}

\section{The Schur-monotone property}\label{sec:Monotone}

We first recall the majorization partial order on $\mathbb{R}^n$. Suppose $x, y\in \mathbb{R}^n$ and denote by $\widetilde{x} = (\widetilde{x}_1, \widetilde{x}_2,\dots , \widetilde{x}_n)$ the decreasing rearrangement of $x$. We say $x$ is \emph{majorized} by $y$ if $\sum_{j=1}^k \widetilde{x}_j\leq \sum_{j=1}^k \widetilde{y}_j$ for $k=1, 2, \ldots, n$ with equality when $k=n$. Moreover, we write $x\prec y$ whenever $x$ is majorized by $y$. This section studies distributions that satisfy a ``Schur-monotone property,'' which we define below.

\begin{definition}
Fix $p\geq 1$, and let $a,b\in \mathbb{R}^n$. Let $X$ be a random variable and $X_1, X_2, \ldots, X_n$ iid copies of $X$. We say $X$ has the \emph{Schur-concave property for $p$} if 

\begin{align*}
(a_1^2, a_2^2,\ldots, a_n^2)\prec (b_1^2, b_2^2, \ldots, b_n^2)\Longrightarrow \norm{ \sum_i a_i X_i}_p\geq \norm{ \sum_i b_i X_i}_p,
\end{align*}in which $\norm{\cdot}_p=(\mathbb{E}|\cdot |^p)^{\frac{1}{p}}$. We say $X$ has the \emph{Schur-convex property for $p$}\footnote{The implication $a\prec b\Longrightarrow \norm{ \sum_i a_i X_i}_p\leq \norm{ \sum_i b_i X_i}_p$ always holds for $p\geq 1$; see \cite{Chavez1}, for example.} if 

\begin{align*}
(a_1^2, a_2^2,\ldots, a_n^2)\prec (b_1^2, b_2^2, \ldots, b_n^2)\Longrightarrow \norm{ \sum_i a_i X_i}_p\leq \norm{ \sum_i b_i X_i}_p.
\end{align*}
\end{definition}

Historically, the interest in the Schur-monotone property lies in its connection with Khintchine inequalities for the underlying distribution. We describe this connection in Section \ref{sec:Khintchine}. Indeed, several distributions are known to have the Schur-monotone property. For instance, suppose $X$ is uniformly distributed on $[-1, 1]$. A classic result due to Lata\l a and Oleszkiewicz \cite{Latala1} ensures that $X$ satisfies the Schur-concave property for $p\geq 2$ and the Schur-convex property for $1\leq p\leq 2$. This result of Lata\l a and Oleszkiewicz has been generalized to random vectors uniformly distributed on Euclidean balls by Eskenazis, Nayar and Tkocz \cite{Eskenazis1}. If $X$ is a Gaussian mixture, then $X$ satisfies the Schur-concave property for $-1<p\leq 2$ and the Schur-convex property for $p\geq 2$. This result is also due to Eskenazis, Nayar and Tkocz \cite{Eskenazis1}. Finally, let $q\in [0, \frac{1}{2}]$ and $X$ be a random variable satisfying $\pr(X=0)=q$ and $\pr(X=\pm 1)=\frac{1-q}{2}$. A result due to Havrilla and Tkocz \cite{Havrilla1} ensures $X$ satisfies the Schur-concave property for $p\geq 3$.

\subsection{Linear combinations}

Our next theorem ensures the Schur-monotone property is preserved under linear combinations. We start with a lemma.

\begin{lemma}\label{maj_linear}
Suppose $a,b\in\mathbb R^n$ satisfy $a \prec b$. Given scalars $\ell_1,\ell_2\dots,\ell_k \in\mathbb{R}$, define $a',b'\in \mathbb{R}^{kn}$ by $a'=(\ell_1a,\ell_2a,\dots,\ell_ka)$ and $(\ell_1b,\ell_2b,\dots,\ell_kb)$. Then $a'\prec b'$.
\end{lemma}

\begin{proof}
The Hardy-Littlewood-P\'olya theorem \cite{Marshall} states that a vector $a\prec b$ if and only if there exists a doubly stochastic matrix $P$ such that $a= Pb$. Therefore, $a=Pb$ for some doubly stochastic matrix $P$. Define the block-diagonal matrix
\begin{align*}
Q=\operatorname{diag}(P,P\dots,P)
\end{align*}
with $k$ copies of $P$ on the diagonal. Since $P$ is doubly stochastic, so is $Q$. Moreover,
\begin{align*}
a' = Qb'.
\end{align*}
Thus, by the Hardy-Littlewood-P\'olya theorem, $a'\prec b'$.
\end{proof}

\begin{theorem}\label{thm:Main1}
Let $p\geq 1$, $X$ a symmetric random variable and $X_1, X_2, \ldots, X_k$ iid copies of $X$. Furthermore, suppose $\alpha_1, \alpha_2, \ldots, \alpha_k\in \mathbb{R}$.
\begin{enumerate}
\item If $X$ has the Schur-concave property for $p$, then so does $Y=\sum_i\alpha_iX_i$;\\

\item If $X$ has the Schur-convex property for $p$, then so does $Y=\sum_i\alpha_iX_i$.
\end{enumerate}
\end{theorem} 

\begin{proof}
Assume $X$ satisfies the Schur-concave property for $p\geq 1$. Consider 
\begin{align*}
Y=\sum_{j=1}^k \alpha_j X_j,
\end{align*} in which $\alpha_1, \alpha_2, \ldots, \alpha_k\in\mathbb{R}$ and $X_1, X_2, \ldots, X_k$ are iid copies of $X$. Define the vector $u = (\alpha_1 a,\alpha_2 a,\dots,\alpha_k a)\in\mathbb R^{kn}$ given $a=(a_1, a_2\dots,a_n)\in\mathbb R^n$. Observe that the sums $\sum_{i=1}^n a_i Y_i$ and $\sum_{r=1}^{kn} u_r X_r$ have the same distribution. Similarly, we define $v= (\alpha_1 b,\alpha_2 b,\dots,\alpha_k b)\in\mathbb R^{kn}$ given $b=(b_1, b_2, \ldots, b_n)\in \mathbb{R}^n$. Then the sums $\sum_{i=1}^n b_i Y_i$ and $\sum_{r=1}^{kn} v_r X_r$ have the same distribution. 

Suppose $(a_1^2,a_2^2\ldots,a_n^2)\prec (b_1^2, b_2^2\ldots,b_n^2)$. Let $\ell_j=\alpha_j^2$ in Lemma \ref{maj_linear} to conclude $(u_1^2, u_2^2, \ldots, u_{kn}^2)\prec (v_1^2, v_2^2,\ldots, v_{kn}^2)$. The Schur-concave property of $X$ implies
\begin{align*}
\left\|\sum_{r=1}^{kn} \widetilde u_r X_r\right\|_p
\ge
\left\|\sum_{r=1}^{kn} \widetilde v_r X_r\right\|_p.
\end{align*}
Equivalently,
\begin{align*}
\left\|\sum_{i=1}^n a_i Y_i\right\|_p
\ge
\left\|\sum_{i=1}^n b_i Y_i\right\|_p.
\end{align*}
The inequalities are reversed if $X$ has the Schur-convex property.

\end{proof}

As an example of Theorem \ref{thm:Main1}, suppose $X$ follows a triangular distribution on $[-1, 1]$. Then we may write $X=\frac{1}{2}U_1+\frac{1}{2}U_2$, in which $U_1$ and $U_2$ are iid uniform random variables on $[-1,1]$. Theorem \ref{thm:Main2} and the result of Lata\l a and Oleszkiewicz, which is highlighted above, imply $X$ satisfies the Schur-concave property for $p\geq 2$ and the Schur-convex property for $1\leq p\leq 2$.

\subsection{Log-monotonicity} Our next theorem establishes the Schur-monotone property for even integers for symmetric distributions whose moment sequences, after normalization, are log-monotone. Before stating the theorem, we review some terminology about weak partitions and Schur-monotonicity of functions.

A \emph{weak partition} of $j$ into $n$ parts is a sequence $\alpha=(\alpha_1, \alpha_2, \ldots, \alpha_n)$ of nonnegative integers for which $|\alpha|=\sum_k\alpha_k=j$. The number of nonzero parts of $\alpha$ is called its \emph{length} and denoted $\ell(\alpha)$. The multinomial theorem ensures that
\begin{align*}
\Big(\sum_{k=1}^n a_kx_k\Big)^j=j!\sum_{|\alpha|=j} \prod_i \frac{a_i^{\alpha_i}x_i^{\alpha_i}}{\alpha_i !},
\end{align*}in which the sum is taken over all weak partitions of $j$ into $n$ parts.

Let $\Omega\subset \mathbb{R}^n$. A function $f:\Omega\to \mathbb{R}$ is \emph{Schur-convex} if $a\prec b$ implies $f(a)\leq f(b)$ for all $a,b\in \Omega$. We say $f$ is \emph{Schur-concave} if $-f$ is Schur-convex. The following lemma plays an important role for us and can be found, for instance, in \cite{Marshall}.

\begin{lemma}[Schur-Ostrowski criterion]\label{lem:Schur}
Let $\Omega\subset\mathbb{R}^n$ be convex. A continuously differentiable function $f:\Omega\to \mathbb{R}$ is Schur-convex if and only if for all $1\leq i, j,\leq n$, 
\begin{align*}
(a_i-a_j)\Big( \frac{\partial f}{\partial a_i}-\frac{\partial f}{\partial a_j}\Big)\geq 0.
\end{align*} 
\end{lemma}

\begin{theorem}\label{thm:Main2}
Let $X$ be a symmetric random variable and $r_k=r_k(X)$. 

\begin{enumerate}
\item If $(r_k)$ is log-concave, then $X$ has the Schur-concave property for any even $p\geq 2$.\\

\item If $(r_k)$ is log-convex, then $X$ has the Schur-convex property for any even $p\geq 2$.
\end{enumerate}
\end{theorem}

\begin{remark}
Recall $(r_k)$ is log-concave if $X$ is ultra sub-Gaussian. Theorem \ref{thm:Main2} ensures that ultra sub-Gaussian random variables satisfy the Schur-concave property for even integers $p\geq 2$. Type-$\mathscr{L}$ random variables are ultra sub-Gaussian. This was established in \cite{Havrilla2}. Consequently, Theorem \ref{thm:Main2} implies all type-$\mathscr{L}$ random variables satisfy the Schur-concave property for even integers $p\geq 2$. 
\end{remark}

\begin{remark}
The authors are unaware of any random variable $X$ that satisfies the Schur-monotone property for some even $p\geq 2$ but for which $(r_k)$ is not log-monotone. We leave any possible classification theorem as an open problem.
\end{remark}

\subsection{Proof of Theorem \ref{thm:Main2}}
We first prove (a). Write $p=2q$ for some integer $q\geq 1$. For $a=(a_1, a_2\dots,a_n)\in \mathbb R_{\geq 0}^n$, define $f(a)=\sum_k \sqrt{a_k}\,X_k$. It suffices to prove 
\begin{align}
a\prec b \Longrightarrow \norm{f(a)}_{p}\ge \|f(b)\|_{p}\label{eq:Ineq1}
\end{align} for all $a,b\in \mathbb{R}^n_{\geq 0}$ by symmetry of $X$. Expanding the $p$th moment of $f(a)$ yields
\begin{align*}
\mathbb E[f(a)^{p}]
&=
\mathbb E\Bigl(\sum_{k=1}^n\sqrt{a_k}\,X_k\Bigr)^{2q}\\
&=
p!
\sum_{|\alpha|=q}
\prod_{i=1}^n
\frac{\mathbb E[X^{2\alpha_i}]}{(2\alpha_i)!}\,a_i^{\alpha_i}\\
&=
p!
\sum_{|\alpha|=q}
\prod_{i=1}^n \frac{r_{\alpha_i}}{\alpha_i!}a_i^{\alpha_i}.
\end{align*}
Recall $H:\mathbb{R}^n_+\to \mathbb{R}$ is \emph{Schur-concave} if $a\prec b$ implies $H(a)\geq H(b)$. We claim
\begin{align*}
H(a)=\sum_{|\alpha|=q}
\prod_{i=1}^n \frac{r_{\alpha_i}}{\alpha_i!}a_i^{\alpha_i}
\end{align*} is Schur-concave. The polynomial $H$ is continuously differentiable on the positive orthant $\mathbb R_+^n$. Therefore, by Lemma \ref{lem:Schur}, it suffices to prove 
\begin{align}
(a_i-a_j)\left(\frac{\partial H}{\partial a_i}-\frac{\partial H}{\partial a_j}\right)\le0\label{eq:Ineq2}
\end{align} for every $i\neq j$. By symmetry, it is enough to consider $i=1$ and $j=2$. Write $x=a_1$, $y=a_2$ and keep $a_3, a_4,\dots,a_n$ fixed. We decompose $H$ according to the exponents of the remaining variables. We have
\begin{align*}
H(x,y,a_3,\dots,a_n)
=
\sum_{|\alpha|=q}
\left(\prod_{i=3}^n \frac{r_{\alpha_i}}{\alpha_i!}a_i^{\alpha_i}\right)
\sum_{u+v=\alpha_1+\alpha_2}\frac{r_u}{u!}\frac{r_v}{v!}\,x^{u} y^{v}.
\end{align*}
Differentiating with respect to $x$ and $y$ gives
\begin{align*}
\frac{\partial H}{\partial x}-\frac{\partial H}{\partial y}
=
\sum_{|\alpha|=q}
\left(\prod_{i=3}^n \frac{r_{\alpha_i}}{\alpha_i!}a_i^{\alpha_i}\right)
D_t(x,y),
\end{align*}
in which $t=\alpha_1+\alpha_2-1$ and
\begin{align*}
D_t(x,y)
=
\sum_{j=0}^{t}
\Bigg(\frac{r_jr_{t-j}}{j!(t-j)!}\Big(\frac{r_{j+1}}{r_j}-\frac{r_{t-j+1}}{r_{t-j}}\Big)\Bigg)
x^j y^{\,t-j}.
\end{align*}
Define $d_j$ to be the coefficient of the term $x^jy^{t-j}$ in $D_t(x,y)$. Since $\prod_{i=3}^n r_{\alpha_i}a_i^{\alpha_i}$ is nonnegative for each $\alpha_i$, it remains to prove that $(x-y)D_t(x,y)\le0$.

Suppose $(r_k)$ is log-concave. Then the ratio sequence $\frac{r_{j+1}}{r_j}$ is nonincreasing in $j$. Therefore, $d_j\ge0$ whenever $j\le t-j$ and $d_j\le0$ whenever $j\ge t-j$. This implies the coefficient sequence $(d_0, d_1\dots,d_t)$ has at most one sign change, from
positive to negative. Now consider the one-variable polynomial $P_t(s)=\sum_{j=0}^t d_j s^j$. Note
\begin{align*}
\sum_{j=0}^t \frac{r_{j+1}r_{t-j}}{j!(t-j)!}
&=\sum_{j=0}^t \frac{r_jr_{t-j+1}}{j!(t-j)!},
\end{align*} which implies $P_t(1)=\sum_{j=0}^t d_j=0$. Descartes' rule of signs implies $s=1$ is the unique positive root of $P_t$ . The signs of the coefficients change from positive to negative, which implies that $P_t(s)\le0$ whenever $s\ge1$. 

 We remark the expression $(x-y)D_t(x,y)$ is symmetric in $x$ and $y$. It therefore suffices to consider the case $x\ge y$. If $y=0$, then we have
\begin{align*}
(x-y)D_t(x,y)&=xD_t(x,0)\\
&=\frac{r_t}{t!}\Big( \frac{r_{t+1}}{r_t}-r_1\Big) x^{t+1}\\
&=\frac{1}{t!}\Big( r_{t+1}-r_1r_t\Big) x^{t+1}.
\end{align*} Log-concavity of $(r_k)$ ensures $r_{t+1}-r_1r_t\leq 0$; see Lemma \ref{lem:Key1}. If $x\geq y>0$, then
\begin{align*}
D_t(x,y)=y^tP_t(x/y)\le0.
\end{align*}

The proof for (b) is similar. In particular, the inequalities in \eqref{eq:Ineq1} and \eqref{eq:Ineq2} are reversed. If $(r_k)$ is log-convex, then the ratio sequence $\frac{r_{j+1}}{r_j}$ is nondecreasing in $j$. Therefore, the signs of the coefficient sequence $(d_0, d_1\dots,d_t)$ change from negative to positive. The same calculation shows \(s=1\) is the unique positive root of \(P_t\) with $P_t(s)\ge0$ for $s\ge1$. Establishing $(x-y)D_t(x,y)\geq 0$ follows just as before. By symmetry, it suffices to consider $x\geq y$. The case $y=0$ is immediate from Lemma \ref{lem:Key1}. If $x\geq y>0$, then $D_t(x,y)=y^tP_t(x/y)\geq 0$.

\section{Khintchine-type inequalities}\label{sec:Khintchine}
We now give an alternate proof of Theorem \ref{thm:Main}. As previously mentioned, the Schur-monotone property has been studied because it yields Khintchine-type inequalities.
The Schur-monotone property yields a sharp comparison between the $L_p$ and $L_2$ norms of $S$, which is described in \cite[Corollary 24]{Eskenazis1} as follows. To begin, suppose $X$ satisfies the Schur-concave property for $p\geq 2$. Then we have
\begin{align*}
\textstyle (\frac{1}{n}, \ldots, \frac{1}{n})\prec (\frac{1}{n-1}, \ldots, \frac{1}{n-1}, 0)\Longrightarrow \norm{\frac{X_1+\cdots +X_n}{\sqrt{n}}}_p\geq \norm{\frac{X_1+\cdots +X_{n-1}}{\sqrt{n-1}}}_p, 
\end{align*} which implies $\norm{\frac{X_1+\cdots +X_n}{\sqrt{n}}}_p$ is increasing in $n$. By homogeneity of the $L_p$ norm, we may assume $a\in\mathbb{R}^n$ is a unit vector. The majorizations given by 
\begin{align*}
\textstyle(\frac{1}{n},\frac{1}{n} \ldots, \frac{1}{n})\prec (a_1^2, a_2^2\ldots, a_n^2)\prec (1, 0, \ldots, 0)
\end{align*}
imply $\norm{\frac{X_1+\cdots +X_n}{\sqrt{n}}}_p\geq \norm{S}_p\geq \norm{X_1}_p$. Observe $\mathbb{E}[X]=0$ and $\sigma^2=\norm{X}_2^2$. The central limit theorem implies that $\frac{X_1+\cdots +X_n}{\norm{X}_2\sqrt{n}}$ converges in distribution to a standard normal random variable. Our inequality above therefore becomes
\begin{align*}
\gamma_p\norm{X}_2\geq \norm{S}_p\geq \norm{X}_p,
\end{align*} in which $\gamma_p$ is the $p$th moment of a standard normal. Finally, independence implies $\norm{S}_2=\norm{X}_2$, which yields $\gamma_p\norm{S}_2\geq \norm{S}_p\geq \big( \frac{\norm{X}_p}{\norm{X}_2}\big) \norm{S}_2$. The Schur-convex case is similar. Therefore, Theorem \ref{thm:Main2} implies Theorem \ref{thm:Main}.

\section{Log-convex examples}\label{sec:Examples}

Several examples of type-$\mathscr{L}$ random variables are already given, for instance, in \cite{Havrilla2}. Moreover, these random variables are known to be ultra sub-Gaussian, which implies $(r_k)$ is log-concave. We therefore provide examples of distributions $X$ for which the moment sequence $r_k(X)$ is log-convex. We give two examples and leave it as an interesting open problem to try and find more of them.

\subsection{Gaussian mixtures} We call $X$ a \emph{Gaussian mixture} if $X=YZ$, in which $Y$ is a nonnegative random variable and $Z$ is a standard Gaussian random variable independent of $Y$. We claim $r_k(X)$ is log-convex. Independence implies
\begin{align*}
\mathbb{E}[X^{2k}]&=\mathbb{E}[Y^{2k}] \gamma_{2k}^{2k}=\frac{ (2k)!\mathbb{E}[Y^{2k}]}{2^k k!}.
\end{align*} Therefore, $r_k=r_k(X)=\frac{\mathbb{E}[Y^{2k}]}{2^k}$ is log-convex if and only if $r_k^2\leq r_{k-1}r_{k+1}$, which is equivalent to $\big(\mathbb{E}[Y^{2k}]\big)^2\leq \mathbb{E}[Y^{2k-2}]\mathbb{E}[Y^{2k+2}]$. Cauchy-Schwarz implies
\begin{align*}
\big(\mathbb{E}[Y^{2k}]\big)^2=\Big(\mathbb{E}\big[ Y^{\frac{2(k-1)}{2}} Y^{\frac{2(k+1)}{2}} \big]\Big)^2\leq \mathbb{E}[ Y^{2(k-1)}]\mathbb{E}[Y^{2(k+1)}],
\end{align*}as desired. Therefore, $r_k(X)$ is log-convex when $X$ is a Gaussian mixture.

\begin{remark}
Suppose $X$ has density proportional to $e^{-|x|^{\alpha}}$ for $0< \alpha\leq 2$. Then $X$ is a Gaussian mixture. This fact is established, for instance, in \cite[Section 2]{Eskenazis1}. We remark that a student's $t$ distribution is also a Gaussian mixture \cite[Section 2]{Eskenazis1}.
\end{remark}

\subsection{Signed exponentials of random variables} 
Consider the random variable $X=\epsilon e^Y$, in which $\epsilon$ is a Rademacher random variable and $Y$ is a random variable independent of $\epsilon$ with moment generating function $M(t)=\mathbb{E}[e^{tY}]$. Observe
\begin{align*}
\mathbb{E}[X^{2k}]=\mathbb{E}[\epsilon^{2k}]\mathbb{E}[e^{2kY}]=M(2k)
\end{align*} since $\mathbb{E}[\epsilon^{2k}]=1$ for all $k\geq 0$. Therefore,
\begin{align*}
r_k(X)= \frac{k!\mathbb{E}[X^{2k}]}{(2k)!}=\frac{k! M(2k)}{(2k)!}.
\end{align*} Observe $(r_k)$ is log-convex if and only if 
\begin{align*}
\Big[\frac{k! M(2k)}{(2k)!}\Big]^2\leq \frac{(k-1)! M(2k-2)}{(2k-2)!}\frac{(k+1)! M(2k+2)}{(2k+2)!},
\end{align*} which is equivalent to
\begin{align}
\frac{k!}{(k-1)!} \frac{(2k+2)!}{(2k)!} \leq \frac{ (k+1)!}{k!} \frac{(2k)!}{(2k-2)!} \frac{M(2k-2)M(2k+2)}{M(2k)^2}.\label{eq:Log}
\end{align}Observe \eqref{eq:Log} holds if and only if for all $k\geq 1$, we have
\begin{align*}
\frac{2k+1}{2k-1}\leq \frac{M(2k-2)M(2k+2)}{M(2k)^2}.
\end{align*} 

\begin{example}
Suppose $Y\sim N(0, \sigma^2)$. Then $M(t)=e^{\frac{\sigma^2t^2}{2}}$, which ensures 
\begin{align*}
\frac{M(2k-2)M(2k+2)}{M(2k)^2}=e^{4\sigma^2}.
\end{align*} We remark that $\frac{2k+1}{2k-1}$ is strictly decreasing for $k\geq 1$. Consequently, 
\begin{align}
e^{4\sigma^2}\geq 3\label{eq:LogConvex}
\end{align} implies that $r_k(X)$ is log-convex. Observe \eqref{eq:LogConvex} holds whenever $\sigma^2\geq \frac{\log(3)}{4}$. 
\end{example}

\section{Proof of Theorem \ref{thm:Robust1}}\label{sec:Robust1}

\begin{lemma}\label{lem:Key0} Let $r_k=r_k(X)$ for a random variable $X$. Define $f:\mathbb{Z}_{\geq 0}^n\to \mathbb{R}$ by setting 
\begin{align*}
f(\alpha_1, \alpha_2, \ldots, \alpha_n)=\prod_ir_{\alpha_i}.
\end{align*}
\begin{enumerate}
\item If $(r_k)$ is log-concave, then $f$ is Schur-concave;\\

\item If $(r_k)$ is log-convex, then $f$ is Schur-convex.
\end{enumerate}
\end{lemma}

\begin{proof}
We first prove (a). Muirhead's lemma \cite[Lemma B.1]{Marshall} implies $\alpha\prec \beta$ if and only if $\alpha$ can be obtained from $\beta$ from a finite sequence of the transforms
\begin{align*}
\beta\mapsto\widetilde{\beta}= (\beta_1, \ldots, \beta_{j-1}, \beta_k+1, \beta_{j+1}, \ldots, \beta_{k-1}, \beta_j-1, \beta_{k+1}, \ldots, \beta_n).
\end{align*} It therefore suffices to prove $f(\beta)\leq f(\widetilde{\beta})$. This inequality holds if and only if 
\begin{align}
r_{\beta_j}r_{\beta_k}\leq r_{\beta_k+1}r_{\beta_j-1}\label{eq:Unit1}
\end{align} Define $s_k=\frac{r_k}{r_{k-1}}$ for $k\geq 1$. Then \eqref{eq:Unit1} is equivalent to
\begin{align}
s_{\beta_j}\leq s_{\beta_k+1}.\label{eq:Unit2}
\end{align}However, we can assume $\beta_j\geq \beta_k+1$ in $\beta$. For otherwise, $\beta$ cannot possibly majorize $\widetilde{\beta}$. The sequence $(s_k)$ is nonincreasing, which implies \eqref{eq:Unit2} holds. The proof of (b) is identical, but the inequalities in \eqref{eq:Unit1} and \eqref{eq:Unit2} are reversed.
\end{proof}

\begin{lemma}\label{lem:Key1}Let $r_k=r_k(X)$ for a random variable $X$. Let $\alpha=(\alpha_1, \alpha_2, \ldots )$ be a weak partition of $q$ with at least two nonzero parts. The following bounds hold.

\begin{enumerate}
\item If $(r_k)$ is log-concave, then $\prod_ir_{\alpha_i}\geq r_1r_{q-1}\geq r_q$;\\

\item If $(r_k)$ is log-convex, then $\prod_ir_{\alpha_i}\leq r_1r_{q-1}\leq r_q$.
\end{enumerate} 
\end{lemma}

\begin{proof}
Define $f: \mathbb{Z}_{\geq 0}^n\to \mathbb{R}$ by setting $f(\alpha_1, \alpha_2, \ldots, \alpha_n)= \prod_i r_{\alpha_i}$. If $\alpha$ has at least two nonzero parts, then $\alpha\prec (q-1, 1, 0, \ldots, 0)\prec (q, 0, \ldots, 0)$. If $(r_k)$ is log-concave, then $f$ is Schur-concave by Lemma \ref{lem:Key0}. Therefore, 
\begin{align*}
f(\alpha)\geq f(q-1, 1, 0, \ldots, 0)\geq f(q, 0, \ldots, 0),
\end{align*} which is equivalent to the inequalities
\begin{align*}
\prod_i r_{\alpha_i}\geq r_1r_{q-1} \geq r_q.
\end{align*} This establishes (a). If $(r_k)$ is log-convex, then $f$ is Schur-convex by Lemma \ref{lem:Key0}. In this case, the above inequalities are reversed, which establishes (b).
\end{proof}

\subsection{Proof of (a)} We prove part (a) of Theorem \ref{thm:Robust1}. Suppose $\alpha=(\alpha_1, \alpha_2, \ldots, \alpha_n)$ is a weak partition of $q$ into $n$ parts. If $\ell(\alpha)=1$, then $\alpha$ contains one nonzero part. In this case, $\alpha=qe_j$ for some $j$, which implies $\prod_i r_{\alpha_i}=r_q$. If $\ell(\alpha)\geq 2$, then $\alpha$ contains at least two nonzero parts. Lemma \ref{lem:Key1} part (a) implies $\prod_ir_{\alpha_i}\geq r_1r_{q-1}$. We can expand $\norm{S}_p^p=\norm{\sum_ka_kX_k}_p^p$ as in the proof of Theorem \ref{thm:Main2} to obtain
\begin{align*}
\norm{\sum_ka_kX_k}_p^p&=p!\sum_{|\alpha|=q}\prod_i \frac{r_{\alpha_i}}{\alpha_i!} a_i^{2\alpha_i}\\
&=p!\sum_{\ell(\alpha)=1}\prod_i \frac{r_{\alpha_i}}{\alpha_i!} a_i^{2\alpha_i}+p!\sum_{\ell(\alpha)\geq 2}\prod_i \frac{r_{\alpha_i}}{\alpha_i!} a_i^{2\alpha_i}\\
&\geq p!r_q\sum_{\ell(\alpha)=1}\prod_i \frac{a_i^{2\alpha_i}}{\alpha_i!} +p!(r_1r_{q-1})\sum_{\ell(\alpha)\geq 2}\prod_i \frac{a_i^{2\alpha_i}}{\alpha_i!}.
\end{align*} We now add and subtract $r_q$ from the term $r_1r_{q-1}$ to obtain
\begin{align}
\norm{\sum_ka_kX_k}_p^p&\geq p!r_q\sum_{\ell(\alpha)=1}\prod_i \frac{a_i^{2\alpha_i}}{\alpha_i!} +p!(r_q+r_1r_{q-1}-r_q)\sum_{\ell(\alpha)\geq 2}\prod_i \frac{a_i^{2\alpha_i}}{\alpha_i!}\notag \\
&=p!r_q\sum_{|\alpha|=q}\prod_i \frac{a_i^{2\alpha_i}}{\alpha_i!} +p!(r_1r_{q-1}-r_q)\sum_{\ell(\alpha)\geq 2}\prod_i \frac{a_i^{2\alpha_i}}{\alpha_i!}.\label{eq:1}
\end{align} The multinomial theorem ensures $q!\sum_{|\alpha|=q}\prod_i \frac{a_i^{2\alpha_i}}{\alpha_i!}=(a_1^2+a_2^2+\cdots +a_n^2)^q$. Therefore, our normalization $\sum_ka_k^2=1$ implies 
\begin{align}
\sum_{|\alpha|=q}\prod_i \frac{a_i^{2\alpha_i}}{\alpha_i!}=\frac{1}{q!}.\label{eq:2}
\end{align} Relation \eqref{eq:2} implies
\begin{align}
\sum_{\ell(\alpha)\geq 2}\prod_i \frac{a_i^{2\alpha_i}}{\alpha_i!}&=\sum_{|\alpha|=q}\prod_i \frac{a_i^{2\alpha_i}}{\alpha_i!}-\sum_{\ell(\alpha)=1}\prod_i \frac{a_i^{2\alpha_i}}{\alpha_i!}\notag \\
&=\frac{1}{q!}-\frac{1}{q!}\sum_k a_k^p\notag\\
&=\frac{1}{q!}\Big( 1- \sum_ka_k^p\Big).\label{eq:3}
\end{align} Observe $\frac{p! r_q}{q!}=\norm{X}_p^p$. Therefore, relations \eqref{eq:1}, \eqref{eq:2} and \eqref{eq:3} imply
\begin{align}
\norm{\sum_ka_kX_k}_p^p\geq \norm{X}_p^p+\textstyle\frac{p!}{q!}(r_1r_{q-1}-r_q)\Big( 1- \displaystyle\sum_ka_k^p\Big).\label{eq:4}
\end{align}Finally, observe $\gamma_p^p=2^{-q} \frac{p!}{q!}$ and $\norm{X}_p^p=\big(\frac{\norm{X}_p}{\norm{X}_2}\big)^p\norm{S}_2^p$ since $\norm{S}_2=\norm{X}_2$. We remark that $r_1r_{q-1}-r_q\geq 0$ by Lemma \ref{lem:Key1}. This establishes our inequality.

We now prove that our inequality is asymptotically sharp near the coordinate extremizers. If $q=1$, then the defect term vanishes. Assume $q\geq 2$. In this case, let $a(t)=(\sqrt{1-t^2}, t,0, \ldots, 0)$ for small $t$. Observe
\begin{align}
1-\sum_ka_k(t)^p=1-(1-t^2)^q-t^{2q}=qt^2+O(t^4).\label{eq:Sharp}
\end{align}On the other hand,  a binomial expansion yields
\begin{align*}
\norm{S}_p^p&=\mathbb{E}| \sqrt{1-t^2}X_1+tX_2|^p\\
&=\sum_{k=0}^q \binom{p}{2k}(1-t^2)^{q-k}t^{2k}\mathbb{E}[X_1^{2(q-k)}]\mathbb{E}[X_2^{2k}].
\end{align*} Only $k=0,1$ contribute constant and quadratic terms. Therefore,
\begin{align*}
\norm{S}_p^p&=\textstyle(1-t^2)^q\mathbb{E}[X^{2q}]+\binom{p}{2}(1-t^2)^{q-1}t^2\mathbb{E}[X^{2(q-1)}]\mathbb{E}[X^2]+O(t^4)\\
&=\textstyle(1-t^2)^q\frac{(2q)!r_q}{q!}+q(2q-1)(1-t^2)^{q-1}t^2\frac{2!(2q-2)!r_1r_{q-1}}{(q-1)!}+O(t^4)\\
&=\textstyle(1-t^2)^q\frac{(2q)!r_q}{q!}+(1-t^2)^{q-1}t^2\frac{(2q)!r_1r_{q-1}}{(q-1)!}+O(t^4)\\
&=\textstyle \frac{(2q)! r_q}{q!}+\Big[\textstyle \frac{(2q)!r_1r_{q-1}}{(q-1)!}-q\frac{(2q)!r_q}{q!}+\Big]t^2+O(t^4)\\
&=\textstyle \mathbb{E}[X^{2q}]+\frac{(2q)!}{(q-1)!}\big( r_1r_{q-1}-r_q\big) t^2+O(t^4)\\
&=\norm{X}_p^p+qC_pt^2+O(t^4),
\end{align*}in which $C_p=2^q\gamma_p^p(r_1r_{q-1}-r_{q})$. In light of \eqref{eq:Sharp}, our lower bound reads
\begin{align*}
\norm{X}_p^p+C_p\Big(1-\sum_ka_k^p\Big)=\norm{X}_p^p+qC_pt^2+O(t^4).
\end{align*}Therefore, our inequality is asymptotically sharp near extremizers.

\subsection{Proof of (b)} The proof of Theorem \ref{thm:Robust1} part (b) is nearly identical to the proof above. However, applying Lemma \ref{lem:Key1} part (b) reverses the inequality in \eqref{eq:4}:
\begin{align}
\norm{\sum_ka_kX_k}_p^p\leq \norm{X}_p^p-\textstyle\frac{p!}{q!}(r_q-r_1r_{q-1})\Big( 1- \displaystyle\sum_ka_k^p\Big).\label{eq:5}
\end{align} Once again, observe $\gamma_p^p=2^{-q} \frac{p!}{q!}$ and $\norm{X}_p^p=\big(\frac{\norm{X}_p}{\norm{X}_2}\big)^p\norm{S}_2^p$. We remark that $r_q-r_1r_{q-1}\geq 0$ by Lemma \ref{lem:Key1}. Optimality of the constants follows as before.

\section{Proof of Theorem \ref{thm:Robust2}}\label{sec:Robust2}

\begin{lemma}\label{lem:Key2}Let $r_k=r_k(X)$ for a random variable $X$. Let $\alpha=(\alpha_1, \alpha_2, \ldots, \alpha_n)$ be a weak partition of $q$ with length $\ell(\alpha)<q$. The following bounds hold.

\begin{enumerate}
\item If $(r_k)$ is log-concave, then $\prod_ir_{\alpha_i}\leq r_1^{q-2}r_2$;\\

\item If $(r_k)$ is log-convex, then $\prod_ir_{\alpha_i}\geq r_1^{q-2}r_2$.
\end{enumerate} 
\end{lemma}

\begin{proof}
Define $f: \mathbb{Z}_{\geq 0}^n\to \mathbb{R}$ by setting $f(\alpha_1, \alpha_2, \ldots, \alpha_n)= \prod_i r_{\alpha_i}$. If $\alpha$ has length $\ell(\alpha)<q$, then $\alpha_i\geq 2$ for some $i$. It follows that 
\begin{align*}
(2, 1, 1, \ldots, 1)\prec \alpha,
\end{align*} in which 1 appears $q-2$ times and 0 is appended on either side as needed. If $(r_k)$ is log-concave, then $f$ is Schur-concave by Lemma \ref{lem:Key0}. Therefore, 
\begin{align*}
f(2,1,1,\ldots, 1)\geq f(\alpha)\Longleftrightarrow \prod_i r_{\alpha_i} \leq r_1^{q-2}r_{2}.
\end{align*} This establishes (a). If $(r_k)$ is log-convex, then $f$ is Schur-convex by Lemma \ref{lem:Key0}. In this case, the above inequalities are reversed, which establishes (b).
\end{proof}

\begin{lemma}\label{lem:Power4}
Let $a=(a_1, a_2, \ldots, a_n)\in \mathbb{R}^n$ satisfy $\sum_ka_k^2=1$. For each integer $q\geq 2$, 
\begin{align*}
q!\sum_{\ell(\alpha)<q} \prod_i \frac{a_i^{2\alpha_i}}{\alpha_i!}\geq \Delta_q(a), 
\end{align*}in which
\begin{align*}
0\leq \Delta_q(a)=\begin{cases} 1 & n<q,\\
1-q!\binom{n}{q}\Big( \frac{1-\sum_ka_k^4}{n(n-1)}\Big)^{\frac{q}{2}}& n\geq q.
\end{cases}
\end{align*}
\end{lemma}

\begin{proof}
Relation \eqref{eq:2} implies
\begin{align}
q!\sum_{\ell(\alpha)<q} \prod_i \frac{a_i^{2\alpha_i}}{\alpha_i!}&=q!\sum_{|\alpha|=q} \prod_i \frac{a_i^{2\alpha_i}}{\alpha_i!}-q!\sum_{\ell(\alpha)=q} \prod_i \frac{a_i^{2\alpha_i}}{\alpha_i!}\notag\\
&=1-q!\sum_{\ell(\alpha)=q} \prod_i \frac{a_i^{2\alpha_i}}{\alpha_i!}.\label{eq:Elementary1}
\end{align}If $\ell(\alpha)=q$, then $\alpha_i=0,1$ for each $i$. Therefore, \eqref{eq:Elementary1} implies
\begin{align}
q!\sum_{\ell(\alpha)<q} \prod_i \frac{a_i^{2\alpha_i}}{\alpha_i!}=1-q!e_q(a_1^2, a_2^2, \ldots, a_n^2),\label{eq:Elementary2}
\end{align}in which $e_q$ denotes the elementary symmetric polynomial of degree $q$. If $n<q$, then the elementary symmetric polynomial vanishes and \eqref{eq:Elementary2} becomes
\begin{align*}
q!\sum_{\ell(\alpha)<q} \prod_i \frac{a_i^{2\alpha_i}}{\alpha_i!}=1.
\end{align*} Suppose that $n\geq q$. Let $p_j(x_1, x_2, \ldots, x_n)=\sum_kx_k^j$ denote the $j$th power sum. The fundamental identity
$e_2=\frac{1}{2}(p_1^2-p_2)$ ensures that
\begin{align}
e_2(a_1^2, a_2^2, \ldots, a_n^2)&=\frac{1}{2}\Big( 1-\sum_ka_k^4\Big)\label{eq:Elementary3}
\end{align} Maclaurin's inequality for $e_q$ \cite[Theorem 1.1]{Niculescu}  and \eqref{eq:Elementary3} yield, for $q\geq 2$,
\begin{align*}
e_q(a_1^2, a_2^2, \ldots, a_n^2)&\leq \binom{n}{q}\Bigg( \frac{2e_2( a_1^2, a_2^2, \ldots, a_n^2)}{n(n-1)} \Bigg)^{\frac{q}{2}}\\
&=\binom{n}{q}\Bigg( \frac{1-\sum_ka_k^4}{n(n-1)}\Bigg)^{\frac{q}{2}}.
\end{align*} Combine this inequality with \eqref{eq:Elementary2} to conclude the first inequality. It remains to show $\Delta_q(a)\geq 0$.  Cauchy-Schwarz ensures that
\begin{align*}
\Big(\sum_k a_k^2\Big)^2\leq \sum_ka_k^4\sum_k1=n\sum_ka_k^4\Longrightarrow \sum_ka_k^4\geq \frac{1}{n}
\end{align*}Observe $q!\binom{n}{q}=n(n-1)\cdots (n-q+1)\leq n^q$. Therefore,
\begin{align*}
q!\binom{n}{q}\Bigg( \frac{1-\sum_ka_k^4}{n(n-1)}\Bigg)^{\frac{q}{2}}&\leq n^q\Big( \frac{1-\frac{1}{n}}{n(n-1)}\Big)^{\frac{q}{2}}\\
&=n^q\Big( \frac{1}{n^2}\Big)^{\frac{q}{2}}\\
&=1,
\end{align*}which implies $1-q!\binom{n}{q}\Big( \frac{1-\sum_ka_k^4}{n(n-1)}\Big)^{\frac{q}{2}}\geq 0$. Therefore, $\Delta_q(a)\geq 0$.
\end{proof}

\begin{lemma}
Let $q\geq 2$ be an integer and $a\in \mathbb{R}^n$ satisfy $\sum_ka_k^2=1$. Then for all $n$, 
\begin{align*}
\Delta_q(a)\geq \sum_ka_k^4.
\end{align*}
\end{lemma}
\begin{proof}
The claim holds  when $n<q$ since $\sum_ka_k^4\leq 1$. Assume $n\geq q$. We show
\begin{align*}
1-p_4\geq q!\binom{n}{q}\Big( \frac{1-p_4}{n(n-1)}\Big)^{\frac{q}{2}},
\end{align*}in which $p_4=\sum_ka_k^4$. This is equivalent to
\begin{align}
 q!\binom{n}{q} \frac{(1-p_4)^{\frac{q}{2}-1}}{\big(n(n-1)\big)^{\frac{q}{2}}}\leq 1.\label{eq:Induction}
\end{align}
As seen above, $p_4\geq \frac{1}{n}$, which implies 
\begin{align*}
 q!\binom{n}{q} \frac{(1-p_4)^{\frac{q}{2}-1}}{\big(n(n-1)\big)^{\frac{q}{2}}}&\leq q!\binom{n}{q} \frac{(1-\frac{1}{n})^{\frac{q}{2}-1}}{\big(n(n-1)\big)^{\frac{q}{2}}}\\
&= q!\binom{n}{q} \frac{1}{n^{q-1}(n-1)}\\
&=\frac{ n(n-1)\cdots (n-q+1)}{ n^{q-1}(n-1)  }.
\end{align*}Let $A(q,n)=\frac{ n(n-1)\cdots (n-q+1)}{ n^{q-1}(n-1)  }$. We prove $A(q,n)\leq 1$.  If $q=2$, then $A(q,n)=1$. Observe $A(q+1, n)=\big(\frac{n-q}{n}\big)A(q,n)$, which allows for induction on $q\geq 2$ since $\frac{n-q}{n}\leq 1$. This establishes \eqref{eq:Induction} and proves the claim.
\end{proof}
\subsection{Proof of (a)} We now prove part (a) of Theorem \ref{thm:Robust2}. Assume $X$ has unit variance. In this case, $\mathbb{E}[X^2]=1$, and we have $r_1=\frac{1}{2}$. Observe
\begin{align*}
\gamma_p^p=\mathbb{E}|Z|^p=\frac{p!}{2^q q!}=\frac{p!r_1^q}{q!},
\end{align*} in which $Z$ is a standard normal random variable. Relation \eqref{eq:2} implies
\begin{align*}
\gamma_p^p=\frac{p!r_1^q}{q!}=p!r_1^q\sum_{|\alpha|=q}\prod_i\frac{a_i^{2\alpha_i}}{\alpha_i!}.
\end{align*}Expand $\norm{S}_p^p$ as in the proof of Theorem \ref{thm:Main2} to get
\begin{align}
\gamma_p^p-\norm{S}_p^p&=p!\sum_{|\alpha|=q}r_1^q\prod_i \frac{a_i^{2\alpha_i}}{\alpha_i!}-p!\sum_{|\alpha|=q}\prod_i \frac{r_{\alpha_i} a_i^{2\alpha_i}}{\alpha_i!}\notag\\
&=p!\sum_{|\alpha|=q}\Big( r_1^q-\prod_i r_{\alpha_i}\Big) \prod_i \frac{a_i^{2\alpha_i}}{\alpha_i!}.\label{eq:6}
\end{align} If $\alpha$ is a weak partition of length $q$, then $r_1^q-\prod_i r_{\alpha_i}=0$. Therefore, the summation in \eqref{eq:6} is taken over all weak partitions of $q$ of length $\ell(\alpha)<q$. In this case, $r_1^q-\prod_i r_{\alpha_i}\geq r_1^q-r_1^{q-2}r_2$ by Lemma \ref{lem:Key2} part (a). Therefore, \eqref{eq:6} becomes 
\begin{align}
\gamma_p^p-\norm{S}_p^p\geq p!r_1^{q-2}(r_1^2-r_2)\sum_{\ell(\alpha)<q} \prod_i \frac{a_i^{2\alpha_i}}{\alpha_i!}.\label{eq:7}
\end{align}Log-concavity of $(r_k)$ ensures $r_1^2\geq r_2$. Therefore, we must bound the summation $\sum_{\ell(\alpha)<q} \prod_i \frac{a_i^{2\alpha_i}}{\alpha_i!}$ from below. We apply Lemma \ref{lem:Power4} to \eqref{eq:7} to get
\begin{align*}
\gamma_p^p-\norm{S}_p^p\geq \frac{p!}{q!}r_1^{q-2}(r_1^2-r_2)\Delta_q(a)=2^q\gamma_p^pr_1^{q-2}(r_1^2-r_2)\Delta_q(a).
\end{align*} Solve the above inequality for $\norm{S}_p^p$ to conclude
\begin{align}
\norm{S}_p^p\leq \gamma_p^p-2^q\gamma_p^pr_1^{q-2}(r_1^2-r_2)\Delta_q(a).\label{eq:8}
\end{align}Finally, if the variance of $X$ is such that $\sigma^2\neq 1$, then apply \eqref{eq:8} to the random variable $\widetilde{X}=\sigma^{-1}X$. In this case, \eqref{eq:8} is equivalent to 
\begin{align}
\sigma^{-p}\norm{S}_p^p\leq \gamma_p^p-2^q\gamma_p^p\sigma^{-p}r_1^{q-2}(r_1^2-r_2)\Delta_q(a).\label{eq:9}
\end{align} Multiply both sides of \eqref{eq:9} by $\sigma^p=\norm{S}_2^p$ to conclude the inequality in (a).

We now prove our inequality is sharp near the Gaussian extremizer. Let $n\geq q$ and  $a=(\frac{1}{\sqrt{n}}, \frac{1}{\sqrt{n}}, \ldots, \frac{1}{\sqrt{n}})$. In this case, $\sum_ka_k^4=\frac{1}{n^2}\sum_k 1=\frac{1}{n}$ and
\begin{align*}
\Delta_q(a)&=1-q!\binom{n}{q}\Bigg( \frac{1-\frac{1}{n}}{n(n-1)}\Bigg)^{\frac{q}{2}}\\
&=1-\frac{q!\binom{n}{q}}{n^q}\\
&=\frac{q(q-1)}{2n}+O(n^{-2})\\
&=\binom{q}{2}\sum_ka_k^4+O(n^{-2}).
\end{align*} We may assume $X$ has unit variance. The right side of our inequality satisfies
\begin{align}
\gamma_p^p-C_p\Delta_q(a)=\gamma_p^p-C_p\binom{q}{2}\sum_ka_k^4+O(n^{-2}),\label{eq:Compare}
\end{align}in which $C_p=2^q\gamma_p^pr_1^{q-2}(r_1^2-r_2)$.  We now evaluate $\norm{S_n}_p^p$, in which $S_n=\sum_k \frac{X_k}{\sqrt{n}}$. As seen in the proof of Theorem \ref{thm:Main2}, we have
\begin{align}
\norm{S_n}_p^p&=p!(\sqrt{n})^{-p}\mathbb{E}[(X_1+X_2+\cdots +X_n)^p]\notag\\
&=p!n^{-q}\sum_{|\alpha|=q}\prod_i \frac{r_{\alpha_i}}{\alpha_i!}.\label{eq:Combo}
\end{align}The number of weak partitions of $q$ into $n$ parts with nonzero parts $(\alpha_1, \alpha_2, \ldots, \alpha_{\ell})$ is given by $\frac{n!}{(n-\ell)!\prod_i m_i!}$, in which $m_i$ denotes the multiplicity of $i$. Observe 
\begin{align*}
n^{-q}\frac{n!}{(n-\ell)!\prod_i m_i!}=O(n^{\ell-q}).
\end{align*} We therefore only need the terms in \eqref{eq:Combo} with $\ell(\alpha)=q-1,q$. These correspond to the partition types $(1, 1,\ldots, 1)$ and $(2, 1, \ldots, 1)$. Therefore, \eqref{eq:Combo} becomes
\begin{align}
\norm{S_n}_p^p&=p!n^{-q}\sum_{\ell(\alpha)=q}\prod_i \frac{r_{\alpha_i}}{\alpha_i!}+p!n^{-q}\sum_{\ell(\alpha)=q-1}\prod_i \frac{r_{\alpha_i}}{\alpha_i!}+O(n^{-2})\notag\\
&=p!n^{-q}\frac{n!r_1^q}{(n-q)! q!}+p!n^{-q}\frac{n!r_1^{q-2}r_2}{(n-q+1)!(q-2)!2!}+O(n^{-2})\notag\\
&=p!n^{-q}\binom{n}{q}r_1^q+p!n^{-q+1}\binom{n-1}{q-2}\frac{r_1^{q-2}r_2}{2}+O(n^{-2}).\label{eq:Final}
\end{align}Observe $\binom{n}{q}=\frac{n^q}{q!}-\binom{q}{2}\frac{n^{q-1}}{q!}+O(n^{q-2})$. Therefore, \eqref{eq:Final} becomes
\begin{align}
\norm{S_n}_p^p&=p!n^{-q}\Bigg( \frac{n^q}{q!}-\binom{q}{2}\frac{n^{q-1}}{q!}     \Bigg)r_1^q+p!n^{-q+1}\Bigg(   \frac{(n-1)^{q-2}}{2(q-2)!}\Bigg)r_1^{q-2}r_2+O(n^{-2})\notag\\
&=\frac{p!}{q!} r_1^q-p!\Bigg( \binom{q}{2}\frac{r_1^q}{q!}-\frac{r_1^{q-2}r_2}{2(q-2)!}     \Bigg)n^{-1}+O(n^{-2})\notag\\
&=\frac{p!}{q!} r_1^q-\frac{p!}{q!}\binom{q}{2}r_1^{q-2}\big(r_1^2-r_2\big)n^{-1}+O(n^{-2}).\label{eq:Final2}
\end{align}We have $r_1=\frac{1}{2}$ since $X$ has unit variance. This implies $\frac{p!}{q!}r_1^q=2^{-q}\frac{p!}{q!}=\gamma_p^p$. Moreover, $n^{-1}=\sum_ka_k^4$. Therefore, \eqref{eq:Final2} yields
\begin{align}
\norm{S_n}_p^p=\gamma_p^p-C_p\binom{q}{2}\sum_ka_k^4+O(n^{-2}).\label{eq:Final3}
\end{align} Compare \eqref{eq:Compare} and \eqref{eq:Final3} to complete the proof.

\subsection{Proof of (b)} Assume once again that $X$ has unit variance. Observe \eqref{eq:6} still holds. However, we now apply Lemma \ref{lem:Key2} part (b) to conclude
\begin{align*}
\gamma_p^p-\norm{S}_p^p&\leq p!r_1^{q-2}(r_1^2-r_2)\sum_{\ell(\alpha)<q} \prod_i \frac{a_i^{2\alpha_i}}{\alpha_i!}\notag\\
&=-p!r_1^{q-2}(r_2-r_1^2)\sum_{\ell(\alpha)<q} \prod_i \frac{a_i^{2\alpha_i}}{\alpha_i!}
\end{align*}This is equivalent to 
\begin{align*}
\norm{S}_p^p-\gamma_p^p\geq p!r_1^{q-2}(r_2-r_1^2)\sum_{\ell(\alpha)<q} \prod_i \frac{a_i^{2\alpha_i}}{\alpha_i!}.
\end{align*} Log-convexity of $(r_k)$ ensures $r_1^2\leq r_2$. Therefore, we must once again bound the summation $\sum_{\ell(\alpha)<q} \prod_i \frac{a_i^{2\alpha_i}}{\alpha_i!}$ from below. Lemma \ref{lem:Power4} implies
\begin{align*}
\norm{S}_p^p-\gamma_p^p \geq \frac{p!}{q!}r_1^{q-2}(r_2-r_1^2)\Delta_q(a)=2^q\gamma_p^pr_1^{q-2}(r_2-r_1^2)\Delta_q(a).
\end{align*} Solve the inequality for $\norm{S}_p^p$ and proceed as before to conclude the inequality in  (b). The asymptotic sharpness argument goes  just as before.

\section{Closing remarks, diagonal stability and standard exponentials}\label{sec:Diagonal}
Suppose $r_k$ is log-monotone. We have established new sharp comparisons between the $L_p$ and $L_2$ norms of the weighted sum $S=\sum_ka_kX_k$. Moreover, we investigated the stability of these inequalities. Finally, our paper is among the first to study stability of Khintchine-type inequalities for non-Rademacher distributions. There are potentially two natural next steps to take, which we now highlight.

\subsection{Diagonal stability}

Suppose $a\in \mathbb{R}^n$ is a unit vector. If $X$ satisfies the Schur-concave property for $p$, then $(\frac{1}{n}, \frac{1}{n},\ldots, \frac{1}{n})\prec (a_1^2, a_2^2, \ldots, a_n^2)$ implies $\norm{S}_p\leq \norm{ Y_n}_p$, in which
\begin{align*}
Y_n=\textstyle\frac{1}{\sqrt{n}}(X_1+X_2+\cdots +X_n).
\end{align*} The equality is reversed if $X$ satisfies the Schur-convex property. In light of Theorem \ref{thm:Main2}, we propose the following \emph{diagonal stability} conjecture.

\begin{conjecture}\label{con:Diagonal}
Suppose $p\geq 2$ is an even integer and $(a_1, a_2, \ldots, a_n)\in \mathbb{R}^n$ satisfies $\sum_ka_k^2=1$. Furthermore, suppose $X$ is a symmetric random variable and $X_1, X_2, \ldots, X_n$ are iid copies of $X$. Let $r_k=r_k(X)$. Then there exist positive constants $C_p>0$ such that the following bounds hold for $S=\sum_ka_kX_k$.
\begin{enumerate}
\item If $(r_k)$ is log-concave, then 
\begin{align*}
\norm{S}_p^p\leq \norm{Y_n}_p^p-C_p\sum_k\textstyle (a_k^2-\frac{1}{n})^2;
\end{align*} 

\item If $(r_k)$ is log-convex, then 
\begin{align*}
\norm{S}_p^p\geq \norm{Y_n}_p^p+C_p\sum_k\textstyle (a_k^2-\frac{1}{n})^2.
\end{align*} 
\end{enumerate}
\end{conjecture}
The diagonal stability conjecture was established for the Rademacher distribution by Jakimiuk \cite{Jakimiuk}. They showed that there exist constants $C_p>0$ for which 
\begin{align*}
\mathbb{E}|\sum_ka_k\epsilon_k|^p\leq \mathbb{E}|\textstyle\frac{1}{\sqrt{n}}\displaystyle\sum_k\epsilon_k|^p-C_p\sum_k\textstyle (a_k^2-\frac{1}{n})^2
\end{align*} for every $p\geq 4$. Of course, the summation $\sum_k(a_k^2-\frac{1}{n})^2$ measures the distance of $a$ from the extremizer. We leave Conjecture \ref{con:Diagonal} as an open problem.

\subsection{Standard exponentials}

Our paper of course considers the case when $X$ is symmetric. Recent work \cite{Brazitikos1,Brazitikos2} establishes sharp bounds for the moments  $\norm{\sum_k a_k X_k}_p$, in which $a\in \mathbb{R}^n$ and $X_1, X_2, \ldots, X_n$ are iid standard exponential random variables. For example, the weighted sum $S=\sum_ka_kX_k$ satisfies the bound $\norm{S}_p\geq \gamma_p\sigma(S)$ for every $p\geq 2$. This result was established by Brazitikos, Tang and Tkocz in \cite{Brazitikos2}. Obtaining a stability version of this inequality, or those inequalities appearing in \cite{Brazitikos1}, would be interesting, even in the case $p\geq 2$ is even.


\bibliographystyle{plain}
\bibliography{Khintchine.bib}

@incollection {Branden,
	AUTHOR = {Br\"and\'en, Petter},
	TITLE = {Unimodality, log-concavity, real-rootedness and beyond},
	BOOKTITLE = {Handbook of Enumerative Combinatorics},
	SERIES = {Discrete Math. Appl. (Boca Raton)},
	PAGES = {437--483},
	PUBLISHER = {CRC Press, Boca Raton, FL},
	YEAR = {2015},
}

@misc{Brazitikos1,
author = {Brazitikos, Silouanos and Pandis, Christos},
year = {2025},
title = {Sharp inequalities for symmetric polynomials, {H}unter's conjecture, and moments of exponential random variables},
note = {Preprint at \url{https://arxiv.org/abs/2512.12254}}
}

@misc{Brazitikos2,
author = {Brazitikos, Silouanos and Tang, Colin and Tkocz, Tomasz},
year = {2026},
title = {Moments of sums of exponentials, beyond CHS},
note = {Preprint at \url{https://arxiv.org/abs/2602.03058}}
}

@article {Chavez1,
	AUTHOR = {Ch{\'a}vez, {\'A}ngel and Garcia, Stephan Ramon and Hurley,
	Jackson},
	TITLE = {Norms on complex matrices induced by random vectors},
	JOURNAL = {Canad. Math. Bull.},
	FJOURNAL = {Canadian Mathematical Bulletin. Bulletin Canadien de
	Math\'ematiques},
	VOLUME = {66},
	YEAR = {2023},
	NUMBER = {3},
	PAGES = {808--826},
	ISSN = {0008-4395,1496-4287},
	MRCLASS = {60B20 (05E05 15A60)},
	MRNUMBER = {4651637},
	DOI = {10.4153/s0008439522000741},
}

@article{De,
author = {De, Anindya and Diakonikolas, Ilias and Servedio, Rocca A.},
year = {2016},
pages = {1058--1094},
title = {A robust {K}hintchine inequality, and algorithms for computing optimal constants in {F}ourier analysis and high-dimensional geometry},
volume = {30},
number = {2},
journal = {SIAM J. Discrete Math.}
}

@article{Eskenazis1,
author = {Eskenazis, Alexandros and Nayar, Piotr and Tkocz, Tomasz},
year = {2018},
pages = {2908--2945},
title = {Gaussian mixtures: Entropy and geometric inequalities},
volume = {56},
number = {5},
journal = {Ann. Probab.}
}

@article{Eskenazis2,
author = {Eskenazis, Alexandros and Nayar, Piotr and Tkocz, Tomasz},
year = {2018},
pages = {389--416},
title = {Sharp comparison of moments and the log-concave moment problem},
volume = {334},
journal = {Adv. Math.}
}

@article{Eskenazis3,
author = {Eskenazis, Alexandros and Nayar, Piotr and Tkocz, Tomasz},
year = {2024},
pages = {1161--1185},
title = {Distributional stability of the {S}zarek and {B}all inequalities},
volume = {389},
journal = {Math. Ann.}
}

@article{Eskenazis4,
author = {Eskenazis, Alexandros and Nayar, Piotr and Tkocz, Tomasz},
year = {2024},
pages = {3377--3412},
title = {Resilience of cube slicing in $\ell_p$},
volume = {173},
number = {17},
journal = {Duke Math. J.}
}

@article{Figiel,
	title={Extremal properties
of {R}ademacher functions with applications to the {K}hintchine and {R}osenthal inequalities},
	author={Figiel, Tomasz and Hitczenko, Pawe{\l} and Johnson, W. B. and Schechtman, Gideon and Zinn, Joel},
	journal={Trans. Amer. Math. Soc.},
	volume={349},
	number={3},
	pages={997--1027},
	year={1997}
}

@article{Haagerup,
author = {Haagerup, Uffe},
year = {1981},
pages = {231--283},
title = {The best constants in the {K}hintchine inequality},
volume = {70},
number = {3},
journal = {Studia Math.},
}

@article{Havrilla1,
author = {Havrilla, Alexander and Tkocz, Tomasz},
year = {2021},
month = {11},
pages = {},
title = {Sharp {K}hinchin-type inequalities for symmetric discrete uniform random variables},
volume = {246},
journal = {Isr. J. Math.}
}

@article{Havrilla2,
author = {Havrilla, Alexander and Nayar, Piotr and Tkocz, Tomasz},
year = {2023},
pages = {2429--2445},
title = {Khinchin-type inequalities via {H}adamard's factorisation},
volume= {3},
journal = {Int. Math. Res. Not.}
}

@article{Jakimiuk,
author = {Jakimiuk, Jacek},
title = {Stability of {K}hintchine inequalities with optimal constants between the second and the $p$-th moment for $p\geq 3$},
volume = {32},
journal = {Bernoulli},
number = {3},
pages = {2524 -- 2542},
year = {2026}
}

@article{Jakimiuk2,
author = {Jakimiuk, Jacek and Tang, Colin and Tkocz, Tomasz},
title = {{K}hinchin inequalities for uniforms on spheres with a deficit},
volume = {313},
journal = {Math. Z.},
number = {11},
year = {2026}
}

@inproceedings{Latala1,
  title={A note on sums of independent uniformly distributed random variables},
  author={Lata{\l}a, Rafa{\l} and Oleszkiewicz, Krzysztof},
  booktitle={Colloquium Mathematicum},
  volume={68},
  number={2},
  pages={197--206},
  year={1995},
  organization={Polska Akademia Nauk. Instytut Matematyczny PAN}
}

@article{Latala2,
	title={On the best constant in the {K}hintchine-{K}ahane inequality},
	author={Lata{\l}a, Rafa{\l} and Oleszkiewicz, Krzysztof},
	journal={Studia Math.},
         volume={109},
        pages={101--104},
	year={1994}
}

@article{Melbourne,
	title={Quantitative form of {B}all's cube slicing in $\mathbb{R}^n$ and equality cases in the min-entropy power inequality},
	author={Melbourne, James and Roberto, Cyril},
	journal={Proc. Amer. Math. Soc.},
         volume={150},
	number = {8},
        pages={3595--3611},
	year={2022}
}

@article{Mordhorst,
title={The optimal constants in {K}hintchine's inequality for $2<p<3$},
	author={Mordhorst, Olaf},
	journal={Colloq. Math.},
         volume={147},
        pages={203--216},
	year={2017}
}

@article{Nayar1,
author = {Nayar, Piotr and Oleszkiewicz, Krzysztof},
year = {2012},
number= {2},
pages = {359--371},
title = {Khintchine-type inequalities with optimal constants via ultra-log concavity},
volume = {16},
journal = {Positivity}
}

@InProceedings{Nazarov,
author="Nazarov, Fedor 
and Podkorytov, Anatoliy",
editor="Havin, Victor P.
and Nikolski, Nikolai K.",
title="Ball, Haagerup, and Distribution Functions",
booktitle="Complex Analysis, Operators, and Related Topics",
year="2000",
publisher="Birkh{\"a}user Basel",
address="Basel",
pages="247--267",
isbn="978-3-0348-8378-8"
}

@article{Newman1,
author = {Newman, Charles M.},
year = {1975},
pages = {1--9},
title = {Inequalities for {I}sing models and field theories which obey the {L}ee-{Y}ang theorem},
volume = {41},
journal = {Comm. Math. Phys.}
}

@article{Newman2,
author = {Newman, Charles M.},
year = {1975},
pages = {913--915},
title = {An extension of {K}hintchine's inequality},
volume = {81},
number = {5},
journal = {Bull. Amer. Math. Soc.}
}

@article{Niculescu,
  author  = {Niculescu, Constantin},
  title   = {A new look at {N}ewton's inequalities},
  journal = {J. Inequal. Pure Appl. Math.},
  volume  = {1},
  number   = {2},
  pages    = {Article 17},
  year     = {2000}
}

@incollection{Ole,
  title={Comparison of moments via {P}oincar{\'e}-type inequality},
  author={Oleszkiewicz, Krzysztof},
  booktitle={Advances in stochastic inequalities ({A}tlanta, {GA}, 1997)},
  series={Contemp. Math.},
  volume={234},
  pages={135--148},
  year={1999},
  publisher={Amer. Math. Soc.},
  address={Providence, RI}
}

@article{Pinelis,
author = {Pinelis, Iosif},
title = {Extremal probabilistic problems and {H}otelling's {T} 2 test under a symmetry condition},
journal = {Ann. Statist.},
volume = {22},
number={1},
year = {1994},
pages={357--368}
}

@article{Khintchine,
author = {Khinchin, Aleksandr},
year = {1923},
pages = {109--116},
title = {\"{U}ber dyadische {B}r\"uche},
volume = {18},
journal = {Math. Z.}
}

@incollection{StanleyLog,
  title     = {Log-concave and unimodal sequences in algebra, combinatorics, and geometry},
  author    = {Stanley, Richard P.},
  booktitle = {Graph theory and its applications: East and West (Jinan, 1986)},
  volume    = {576},
  series    = {Annals of the New York Academy of Sciences},
  pages     = {500--535},
  year      = {1989},
  publisher = {New York Academy of Sciences}
}

@article{Stechkin,
author = {Stechkin, S. B.},
year = {1961},
pages = {357--366},
title = {On the best lacunary systems of functions},
volume = {25},
number = {3},
journal = {Izv. Akad. Nauk SSSR Ser. Mat.}
}

@article{Szarek,
author = {Szarek, Stanis{\l}aw},
year = {1976},
pages = {197--208},
title = {On  the best constant in the {K}hintchine inequality},
volume = {58},
journal = {Studia Math.}
}

@article{Whittle,
author = {Whittle, Peter},
year = {1960},
pages = {302--305},
title = {Bounds for the moments of linear and quadratic forms in independent random variables},
volume = {5},
journal = {Theory Probab. Appl.}
}

@article{Young,
	title={On the best possible constants in the {K}hintchine inequality},
	author={Young, R. M. G.},
	journal={J. London Math. Soc.},
	volume={14},
	number={2},
	pages={496--504},
	year={1976}
}

@book{Marshall,
  title={Inequalities: Theory of Majorization and its Applications},
  author={Marshall, Albert and Olkin, Ingram and Arnold, Barry},
  edition={2nd},
  year={2011},
  publisher={Springer New York}
}

\end{document}